\documentclass[11pt, twoside]{amsart}
\usepackage{epsfig, amssymb, amsfonts, amssymb,xspace, amscdx}

\usepackage[latin1]{inputenc}

\usepackage[T1]{fontenc}
\usepackage[all]{xy}

\newtheorem{lemma}{Lemma}[section]
\newtheorem{proposition}[lemma]{Proposition}
\newtheorem{prob}[lemma]{Problem}
\newtheorem{theorem}[lemma]{Theorem}

\newtheorem{corollary}[lemma]{Corollary}

\newtheorem*{special theorem}{My Specially-Named Theorem}

\makeatletter
\@namedef{subjclassname@2020}{Subject}
\makeatother

\begin{document}

% this starts the page-numbering at the bottom-center with roman numerals,
% like you're supposed to for the beginning pages

\pagestyle{plain}

% DOCUMENT INFO

\title{Discrete valuation rings, partitions and $p$-groups I}
\author{Boubakeur Bahri and Yassine Guerboussa}
\date{\today}

% TITLE PAGE
\maketitle

\begin{center}
\textit{In memory of  Mohsen S. Ghoraishi (...-2022)}.
\end{center}

\begin{abstract}
A finite abelian $p$-group having an automorphism $x$ such that $1+\ldots+x^{p-1}=0$, can be viewed as a module over an appropriate discrete valuation ring $\mathcal{O}$ containing $\mathbb{Z}_p$ (the ring of $p$-adic integer).  This yields the natural problem of comparing the invariants of $A$ as a $\mathbb{Z}_p$-module to its invariants as an $\mathcal{O}$-module.   We solve the latter problem in a more general context, and give some  applications to the structure of some $p$-groups and their automorphisms.

\end{abstract}\vspace{1cm}

%\textbf{MSC 2020}: {primary 20D15; secondary 05A17, 13A18} 


\section{Introduction; the main results}
Let $\mathfrak{o}$ and $\mathfrak{O}$ be two discrete valuation rings with maximal ideals $\mathfrak{p}$ and $\mathfrak{P}$ respectively,  and suppose  that we have  an embedding 
$\rho:\mathfrak{o}\to \mathfrak{O}$, for which $\mathfrak{O}$ is a finitely generated  $\mathfrak{o}$-module.  Therefore, the residual field $\mathfrak{o}/\mathfrak{p}$ embeds in $\mathfrak{O}/\mathfrak{P}$, and the corresponding extension's degree $d=[\mathfrak{O}/\mathfrak{P}:\mathfrak{o}/\mathfrak{p}]$  is finite. Moreover, as $\mathfrak{p}\mathfrak{O}\neq \mathfrak{O}$ (by Nakayama's lemma), there exists a unique integer $e\geq 1$ such that 
$\mathfrak{p}\mathfrak{O}=\mathfrak{P}^e$. 

Let us denote by $\mathcal{F}_{\mathfrak{O}}$ (resp. $\mathcal{F}_{\mathfrak{o}}$) the category of finitely generated torsion modules over $\mathfrak{O}$ (resp. over  $\mathfrak{o}$).  For an object  $M$ in $\mathcal{F}_{\mathfrak{O}}$, we shall denote by $\rho_{\ast} M$ the $\mathfrak{o}$-module obtained from $M$ by   restricting the scalars to $\mathfrak{o}$.  Since $\mathfrak{O}$ is finitely generated over $\mathfrak{o}$, and $e\geq 1$, it follows that   $\rho_* M$ lies in $\mathcal{F}_{\mathfrak{o}}$; so we have a canonical additive functor $M\mapsto \rho_* M$, from  $\mathcal{F}_{\mathfrak{O}}$ to $\mathcal{F}_{\mathfrak{o}}$.   Now, as $\mathfrak{O}$ is a principal ideal domain, to  $M$  corresponds a unique partition $\lambda=(n_1,\ldots,n_s)$ such that
$$M\cong \bigoplus_{i=1}^{s} \mathfrak{O}/\mathfrak{P}^{n_i}.$$
Similarly,  there exists a unique partition  $\rho_{\ast}(\lambda)=(m_1,\ldots,m_r)$ such that
	$$\rho_{\ast}M=\bigoplus_{i= 1}^r \mathfrak{o}/\mathfrak{p}^{m_i}.$$
Thus, we have the natural problem of determining the relationship between the invariants $(n_1,\ldots,n_s)$ and $(m_1,\ldots,m_r)$  above;  in other words:
	\begin{prob}\label{The Combinatorial problem} For all $ \mathfrak{o}$, $\mathfrak{O}$, and $\rho:\mathfrak{o}\to \mathfrak{O}$ satisfying the above assumptions, determine the corresponding map $\lambda \mapsto \rho_{\ast}(\lambda)$.
	\end{prob}
 
This problem is motivated by a situation that often occurs  in studying  automorphisms of finite $p$-groups.  The latter is the case of an abelian  finite $p$-group $A$ on which  a cyclic group $\langle x \rangle$ acts  such that
\begin{equation}\label{Basic contraint on modules}
a \cdot (1+x+\cdots+x^{p-1})=0, \quad \mbox{ for all } a\in A.
\end{equation} 
(That is, $x$ induces a \textit{splitting automorphism }of order $p$ of $A$.)

In such a  case, $A$ has a canonical structure of a module over $\mathbb{Z}_p$ (the ring of $p$-adic integers).   Moreover, the action of $x$ on $A$ extends naturally to the polynomial ring  $\mathbb{Z}_p[X]$, with $X$ acts on $A$ as the endomorphism $x-1$, that is  
$a\cdot f=a\cdot f(x-1)$,  for all  $a\in A$ and all $f\in \mathbb{Z}_p[X]$.
\medskip 

Set $g=\sum_{i= 1}^{p}\binom{p}{i}X^{i-1}$. Plainly, we have $g=((X+1)^{p}-1)/X$, so $$g(x-1)=1+x+\cdots+x^{p-1}, \quad    \mbox{ and } \quad A\cdot g=0.$$  
Hence,  $A$ is in fact a module over $\mathfrak{O}=
\mathbb{Z}_p[X]/(g)$.  Moreover, since $g$ is an Eisenstein polynomial,  $\mathfrak{O}$ is a discrete valuation ring with maximal ideal  $\mathfrak{P}=(\pi)$, where $\pi$ denotes the canonical image of $X$ in $\mathfrak{O}$ (cf. \cite[Chap. I, \S 6]{Serre}; see also \S  \ref{Concl Sec}).  To summarize, we have the following fact (which we shall use freely in the sequel). 
\medskip 

\textbf{($*$)} \emph{If  $A$ is a finite abelian $p$-group, and $\langle x \rangle$  is a cyclic group acting on $A$ and satisfying the condition (\ref{Basic contraint on modules}), then $A$ is a module over the discrete valuation ring   $\mathfrak{O}=
\mathbb{Z}_p[X]/(g)$ under the action
$$a\cdot f(\pi)=a\cdot f(x-1), \quad \mbox{ for all }  a\in A \mbox{ and all } f\in \mathbb{Z}_p[X],$$
with $\pi$ denotes the canonical image of $X$ in $\mathfrak{O}$.} 

\medskip

In the above situation, we may compare  the invariants of $A$ as an abelian $p$-group (equivalently, as a $\mathbb{Z}_p$-module) with its invariants as a module over  $\mathbb{Z}_p[X]/(g)$.  In other words, we are led to Problem \ref{The Combinatorial problem}, for $\mathfrak{o}= \mathbb{Z}_p$ and $\mathfrak{O}=
\mathbb{Z}_p[X]/(g)$.  As we shall see below, solving the latter yields  useful restrictions on the structure of $A$, which can be applied in several situations, specially  in studying  the  automorphisms of finite $p$-groups.

\medskip

We start by giving a solution to Problem \ref{The Combinatorial problem};  we need  to introduce some notations. For  a partition  $\lambda=(n_1,\ldots,n_s)$,  we shall write $\ell(\lambda)$ for the length of $\lambda$, and $\vert\lambda\vert$ for the sum of its parts, that is  $\ell(\lambda)=s$, and  $\vert\lambda\vert=\sum_{i= 1}^{s}n_i$. If  $l$ is a positive  integer,  then we denote by $\lambda_{e,l}$   the sub-partition of  $\lambda$ formed by the parts $n_j$ that satisfy $(l-1)e<n_j \leq le$ (recall that $e$ is the positive integer satisfying $\mathfrak{p}\mathfrak{O}=\mathfrak{P}^e$ ).  We  define the weight $w_{e,l}(\lambda)$ of $\lambda$ at $(e,l)$  by  $$w_{e,l}(\lambda)=\sum_{n_j} (n_j- (l-1)e),$$
where $n_j$ ranges  over all the components of $\lambda_{e,l}$.  If $\lambda_{e,l}$ is empty, we set  $\ell(\lambda_{e,j})=\vert\lambda_{e,j}\vert=0$. Clearly, we have 
\begin{equation}\label{The relation between the weight and the sum of the parts}
w_{e,l}(\lambda)= \vert\lambda_{e,l}\vert-\ell(\lambda_{e,l})(l-1)e.
\end{equation}
Finally, we need  the following key ingredient:
\begin{equation}\label{The invariants of the associated abelian group}
f_{e,l}(\lambda)=\ell(\lambda_{e,l+1})e-w_{e,l+1}(\lambda)+w_{e,l}(\lambda).
\end{equation}

\begin{theorem}\label{Main 1}
	Let  $M$ be an object in $\mathcal{F}_{\mathfrak{O}}$, and  $\lambda$ be its corresponding partition.  Then, 
	$$\rho_{\ast}M\cong \bigoplus_{i\geq 1} (\mathfrak{o}/{\mathfrak{p}^i})^{f_{e,i}(\lambda) d}.$$

\end{theorem}
Using the notation $(1^{a_1}\ldots i^{a_i}\dots)$ for the partition where each  $i$ appears $a_i$ times, it follows that  the partition $\rho_{\ast}(\lambda)$ corresponding to $\rho_{\ast}M$ is equal to $$\rho_{\ast}(\lambda)=(1^{f_{e,1}(\lambda)d}\ldots i^{f_{e,i}(\lambda)d}\ldots).$$

\medskip 

Take $\mathfrak{o}=\mathbb{Z}_p$ and $\mathfrak{O}=\mathbb{Z}_p[X]/(g)$, with $g=\sum_{i=1}^{p} \binom{p}{i} X^{i-1}$. We have $\mathfrak{P}=(\pi)$, where $\pi$ is the canonical image of $X$ in $\mathfrak{O}$,  and $p\mathfrak{O}=\mathfrak{P}^{p-1}$ (so $e=p-1$). The residual fields are  $\mathfrak{O}/\mathfrak{P}=\mathfrak{o}/\mathfrak{p}=\mathbb{Z}/p\mathbb{Z}$,  so $d=1$.  Note also that $\mathfrak{o}/\mathfrak{p}^{n}\cong  \mathbb{Z}/p^{n}\mathbb{Z}$, for all $n\geq 1$.  Applying Theorem \ref{Main 1} with these data gives the following.
\begin{corollary}\label{Structure of om as abelian groups}  Let $A$ be a finite abelian $p$-group, and $\langle x \rangle$  a cyclic group acting on $A$ such that $A\cdot (1+\cdots+x^{p-1})=0$.   If $\lambda$ is the partition corresponding to $A$ viewed as a module over  $\mathfrak{O}=\mathbb{Z}_p[X]/(g)$, then 
	$$A\cong \bigoplus_{i\geq 1} \mathbb{Z}/p^{f_{p-1,i}(\lambda)}\mathbb{Z}.$$
In particular, for $\lambda=(n)$,   we have
	$$A\cong (\mathbb{Z}/p^{l+1}\mathbb{Z})^{r} \oplus (\mathbb{Z}/p^{l}\mathbb{Z})^{p-r-1}, $$
	where $n=l(p-1)+r$, and $0< r \leq p-1$. 
\end{corollary}

In fact, one can obtain detailed information on the structure of any extension $H$ of $A$ by $\langle x\rangle$.  In what follows, $\gamma_n(H)$ denotes  the $n$th term of the lower central series of $H$, and $d(H)$ the  minimal number of its generators.

 \begin{theorem}\label{Main 2} 
Let $A$ be a non-trivial finite abelian $p$-group,  $\langle x \rangle$ be a cyclic $p$-group acting on $A$ such that $A\cdot (1+\cdots+x^{p-1})=0$, and $H$ be any extension of $A$ by $\langle x \rangle$.  If $\lambda=(n_1,\ldots,n_s)$ is the partition corresponding to $A$ viewed as a module over  $\mathfrak{O}=\mathbb{Z}_p[X]/(g)$, then 
 	\begin{itemize}
 			\item[(i)] For every $n\geq 1$, we have $\gamma_{n+1}(H)=A\cdot \mathfrak{P}^{n}$.  
 		\item[(ii)]  $H$ is nilpotent of class $n_1$, and $A$ has exponent  $p^{m}$, with
 		$m=\lceil \dfrac{n_1}{p-1} \rceil$ (that is, $m$ is the smallest integer $\geq  \dfrac{n_1}{p-1}$).
 		\item[(iii)] For every $n\geq 1$, we have $A^{p^n}=\gamma_{np-n+1}(H)$ (note that $A^{p^n}$ corresponds to $p^n A$ in the additive notation).
 		\item[(iv)] For every  $j\geq 2$, $\gamma_{j}(H)/\gamma_{j+1}(H)$ is elementary abelian of rank $d_j$, where   $d_j$ is  the number of  parts of $\lambda$ that are $\geq j$.
 			\item[(v)] $A/\gamma_{2}(H)$ is elementary abelian of rank $d_1=\ell(\lambda)$, and   $$d(A)= \sum_{j=1}^{p-1} d_j.$$
 	\end{itemize}
\end{theorem}
Note that   $(d_1,d_2,\ldots)$ is just the conjugate of the partition $\lambda$;  consequently  $\lambda$ can be covered from the factors  of the lower series of $H$.  The terms of the upper central series $Z_n(H)$ can be determiner similarly. For instance, $Z(H)$ is generated by $\langle x^p\rangle$ and the subgroup $A^{\langle x\rangle}$ of $A$ formed by the fixed elements under $x$; one easily sees that
$A^{\langle x\rangle}\cong \bigoplus_{i=1}^{s} \mathfrak{P}^{n_i-1}/\mathfrak{P}^{n_i}$, which is elementary abelian of rank $\ell(\lambda)$.

\medskip 

As an illustration of the utility of Corollary \ref{Structure of om as abelian groups} and Theorem \ref{Main 2}  in studying  automorphisms of finite $p$-groups, we shall prove the following.
\begin{theorem}\label{App to the noninner Conj}
	Let $G$ be a finite non-abelian $p$-group such that   $Z(\Phi(G))$, the center of the Frattini subgroup, can be generated by $p-2$ elements.   Then $G$ has a non-inner automorphism of order $p$ that fixes $\Phi(G)$ element-wise.
\end{theorem}
The above confirms in particular   \textit{the non-inner conjecture} (see \cite[Problem 4.13]{Mazurov}): \textit{every finite non-abelian $p$-group has non-inner automorphisms of order $p$},  for another class of $p$-groups.   For more  information and references about this conjecture see  Ghoraishi \cite{Ghoraishi}.  In fact,  Theorem \ref{App to the noninner Conj}   improves the main result in \cite{Ghoraishi}, which shows that the conjecture holds for $p$-groups with $\vert \Phi(G)\vert \leq p^5$.  Combining our result with the fact that the conjecture holds when $\Phi(G)$ is abelian, it follows that it holds  when $\vert \Phi(G)\vert \leq p^{p}$ (note that the latter bound can be improved to $p^{p+1}$).

Early works of  Gaschütz \cite{Gasch1965,Gasch1966}, and Schmid \cite{Schmid}, have shown a  pertinent relation between $p$-automorphisms and cohomological properties of finite $p$-groups, that can be applied to the non-inner conjecture (see also Abdollahi \cite{Ali1}, and \cite{Benmoussa}).  Our proof of Theorem \ref{App to the noninner Conj} relies on cohomology; the following result will be proved.
\begin{theorem}\label{The cohomological auxilliary result}
	Let $Q$ be a non-cyclic finite $p$-group,  and $A\neq 0$ be a $Q$-module that is also a finite $p$-group.  If $A$ can be generated by $p-2$ elements (as an abelian group), then $A$ is not cohomologically trivial. 
\end{theorem}
We may recall here that the $Q$-module $A$ is said to be  \textit{cohomologically trivial},  if
the Tate cohomology groups $\hat{H}^n(S,A)$
vanish, for all  $n\in \mathbb{Z}$ and all $S\leq Q$. 

A noteworthy is that  Corollary \ref{Structure of om as abelian groups} and Theorem \ref{Main 2} can be pushed further to obtain  results like  Theorems \ref{The cohomological auxilliary result} and \ref{App to the noninner Conj},  although this would remove us far beyond the aim of this note (cf. Section \ref{Cohom and p-Aut}).  We shall pursue that in a separate paper. 
\medskip 

Proofs of Theorems \ref{Main 1} and \ref{Main 2}  are given in \S\ref{Proof of Th1} and \S \ref{p-Groups} respectively.  In \S \ref{Cohom and p-Aut}, we discuss  the main ideas for approaching  \textit{the non-inner conjecture} via cohomology and how we could profit from Corollary \ref{Structure of om as abelian groups}  and Theorem \ref{Main 2} in this context;  proofs of Theorem \ref{App to the noninner Conj} and \ref{The cohomological auxilliary result} are given in that section.  In \S \ref{Concl Sec}, we discuss other  applications, generalizations, and   some combinatorial questions related to Problem \ref{The Combinatorial problem}.

\section{Proof of Theorem \ref{Main 1}}\label{Proof of Th1}
First, note that for all integers $i,j\geq 1$,  the $\mathfrak{o}$-modules $\mathfrak{O}/\mathfrak{P}^j$ and $\mathfrak{P}^{i}/\mathfrak{P}^{j+i}$ are isomorphic (the $\mathfrak{o}$-linear map $x\mapsto \pi^ix$ induces such an isomorphism, where   $\pi$ is a uniformizing element  of $\mathfrak{O}$).

\begin{lemma}\label{Basic lemma} The $\mathfrak{o}$-module
	$\mathfrak{O}$ is  free  of rank $ed$, and 
	$$\mathfrak{O}/\mathfrak{P}^r\cong  (\mathfrak{o}/\mathfrak{p})^{rd}, \quad \mbox{for all } r=1,\ldots,e.$$
(Here, $(\mathfrak{o}/\mathfrak{p})^{n}$ stands for the direct sum of $n$ copies of the $\mathfrak{o}$-module $k$.) 
\end{lemma}
\begin{proof}

By assumption, $\mathfrak{O}$ is a torsion-free finitely generated  $\mathfrak{o}$-module, so there is an integer $n\geq 1$ such that $\mathfrak{O}$ is a direct sum of $n$ copies of  $\mathfrak{o}$.  It follows that $\mathfrak{O}/\mathfrak{p}\mathfrak{O}\cong (\mathfrak{o}/\mathfrak{p})^n$. On the other hand, as $\mathfrak{O}/\mathfrak{P}$ has degree $d$ over $\mathfrak{o}/\mathfrak{p}$, we have $\mathfrak{O}/\mathfrak{P}\cong (\mathfrak{o}/\mathfrak{p})^d$, so $\mathfrak{O}/\mathfrak{P}$ has length $d$ over $\mathfrak{o}$.  Now, as $\mathfrak{P}^{i}/\mathfrak{P}^{i+1}\cong
\mathfrak{O}/\mathfrak{P}$, the $\mathfrak{o}$-module $\mathfrak{O}/\mathfrak{P}^{r}$ has length $rd$, for all $r\geq 1$.  For $1\leq r \leq e$, we have $\mathfrak{p}\mathfrak{O}=\mathfrak{P}^e\subseteq \mathfrak{P}^r$, so $\mathfrak{O}/\mathfrak{P}^r$ is  a quotient of $\mathfrak{O}/\mathfrak{p}\mathfrak{O}\cong (\mathfrak{o}/\mathfrak{p})^n$; thus
$\mathfrak{O}/\mathfrak{P}^r\cong (\mathfrak{o}/\mathfrak{p})^{rd}$, as claimed.  Taking $r=e$, yields at once $n=ed$, which completes the proof. 
\end{proof}

Next, we shall determine the structure of the $\mathfrak{o} $-module $\mathfrak{O}/\mathfrak{P}^{n}$, for all $n\geq 1$.
\begin{proposition}\label{First passage between structures}
	Let $n$ be a positive integer, and write $n=le+r$, with $0<  r \leq e$.  Then,
	$$\mathfrak{O}/\mathfrak{P}^{n}\cong (\mathfrak{o}/\mathfrak{p}^{l+1})^{rd} \oplus (\mathfrak{o}/\mathfrak{p}^{l})^{(e-r)d}, $$
	
\end{proposition}    
\begin{proof}
Firstly, if $l=0$, then $n=r\leq e$, so by Lemma \ref{Basic lemma}, $\mathfrak{O}/\mathfrak{P}^{n}\cong (\mathfrak{o}/\mathfrak{p})^{rd} $, as claimed.
	
	Secondly, if  $n=le$,  then $\mathfrak{P}^{n}=(\mathfrak{p}\mathfrak{O})^l$; so $\mathfrak{O}/\mathfrak{P}^{n}=\mathfrak{O}/\mathfrak{p}^l\mathfrak{O}$, and we have   $\mathfrak{O}/\mathfrak{p}^l\mathfrak{O}\cong ({\mathfrak{o}}/\mathfrak{p}^l)^{ed}$, by the first part of Lemma \ref{Basic lemma}; hence, the result holds in that case.
	
	For the general case, write
	$$\mathfrak{O}/\mathfrak{P}^{n}\cong \mathfrak{o}/\mathfrak{p}^{m_1} \oplus \mathfrak{o}/\mathfrak{p}^{m_2} \oplus \cdots \oplus \mathfrak{o}/\mathfrak{p}^{m_{t}},$$
	with  $m_1\geq m_2 \geq \ldots \geq m_{t}\geq 1$.  We have $\mathfrak{p}^l(\mathfrak{O}/\mathfrak{P}^{n})=\mathfrak{P}^{el}/\mathfrak{P}^{n}$, which is isomorphic to $\mathfrak{O}/\mathfrak{P}^{r}$; so $\mathfrak{p}^l(\mathfrak{O}/\mathfrak{P}^{n})\cong (\mathfrak{o}/\mathfrak{p})^{rd}$,  by the first case.  On the other hand, we have
	$$\mathfrak{p}^l( \bigoplus_{i=1}^{t} \mathfrak{o}/\mathfrak{p}^{m_i})= \bigoplus_{i=1}^{t} \mathfrak{p}^{\min\{l,m_i\}}/\mathfrak{p}^{m_i}\cong \bigoplus_{i=1}^{t} \mathfrak{o}/\mathfrak{p}^{\max\{m_i-l,0\}}. $$
	Thus, $m_i-l=1$ for $i=1,\ldots,rd$; and $m_i\leq l$ for $i=rd+1,\ldots, t$. 
	Moreover, the quotient of  $\mathfrak{O}/\mathfrak{P}^{n}$ by $\mathfrak{p}^l(\mathfrak{O}/\mathfrak{P}^{n})$ is clearly isomorphic to $\mathfrak{O}/\mathfrak{p}^{l}\mathfrak{O}=\mathfrak{O}/\mathfrak{P}^{el}$; so  by the second case, $\mathfrak{O}/\mathfrak{p}^{l}\mathfrak{O}\cong (\mathfrak{o}/\mathfrak{p}^l)^{ed}$.   
	Therefore, 
	$$\bigoplus_{i=1}^{t} \mathfrak{o}/\mathfrak{p}^{\min\{l,m_i\}}\cong (\mathfrak{o}/\mathfrak{p}^l)^{ed}.$$
	Thus, $t=ed$, and $m_i=l$ for all $i=rd+1,\ldots,ed$; the result follows.  
\end{proof}

To complete the proof of Theorem \ref{Main 1}, let $M$ be an object of   $\mathcal{F}_{\mathfrak{O}}$, and $\lambda$ be its corresponding partition.   Assume that $l$ is the largest integer for which $\lambda_{e,l}$ is non-empty, and let $n_1\geq \ldots\geq n_a$ be the components of $\lambda_{e,l}$.  Let $j$ be the largest index for which $n_1=\ldots=n_j=le$; so $n_i=(l-1)e+r_i$, with $0< r_i< e$, for  $i=j+1,\ldots,a$.
.  By Proposition \ref{First passage between structures},  the components of $\lambda_{e,l}$ yield a direct factor of  $\rho_*M$  of the form
$$[(\mathfrak{o}/{\mathfrak{p}^l})^{je}\oplus ((\mathfrak{o}/{\mathfrak{p}^l})^{\sum_{i=j+1}^{a}r_i} \oplus (\mathfrak{o}/\mathfrak{p}^{l-1})^{(a-j)e-\sum_{i=j+1}^{a}r_i})]^d.$$
Now, observe that $je+\sum_{i=j+1}^{a}r_i$ is nothing but the weight $w_{e,l}(\lambda)$, and $a$ is just the length of $\lambda_{e,l}$; so the above  can be written as
$$[(\mathfrak{o}/\mathfrak{p}^l)^{w_{e,l}(\lambda)}\oplus  (\mathfrak{o}/\mathfrak{p}^{l-1})^{\ell(\lambda_{e,l})e-w_{e,l}(\lambda)}]^d.$$
Similarly, $\lambda_{e,l-1}$ induces a direct factor of $\rho_*M$  of the form
$$[(\mathfrak{o}/\mathfrak{p}^{l-1})^{w_{e,l-1}(\lambda)}\oplus  (\mathfrak{o}/\mathfrak{p}^{l-2})^{\ell(\lambda_{e,l-1})e-w_{e,l-1}(\lambda)}]^d.$$
It follows that the number of copies of $\mathfrak{o}/\mathfrak{p}^{l-1} $ involved in the canonical decomposition of $\rho_*M$  is equal to
$f_{e,l-1}(\lambda) d$.  Also, the number of copies of $\mathfrak{o}/\mathfrak{p}^{l}$ involved in  $\rho_*M$ is $w_{e,l}(\lambda)d$, which coincides with $f_{e,l}(\lambda)d$ because
$\ell(\lambda_{e,l+1})=w_{e,l+1}(\lambda)=0$. Arguing in the same way with  $i=l-2,\ldots,1$,  one sees that the number of copies of $\mathfrak{o}/\mathfrak{p}^{i} $ involved in  $\rho_*M$  is equal to $f_{e,i}(\lambda) d$, which concludes the proof.

\section{ Proof of Theorem \ref{Main 2}}\label{p-Groups}
For $a\in A$, the element $a\cdot (x-1)$  coincides clearly with the commutator  $[a,x]$ in $H$.  We claim that
	\begin{equation}\label{Lower series of the cyclotomoc group}
\gamma_{n+1}(H)=A\cdot (x-1)^n,\quad \mbox{for all } n\geq 1.
\end{equation}
The inclusion $A\cdot (x-1)^n \subseteq \gamma_{n+1}(H)$ is obvious.  For the reverse inclusion  we proceed by induction on $n$.  First, note that every  $h\in H$ can be written as $h=bx^{i}$, for some $b\in A$ and some integer $i\geq 1$. Using the fact that $(x^i-1)=(1+\cdots+x^{i-1})(x-1)$, it follows that
$$[a,h]=[a,x^i]=a\cdot (x^i-1) \in [A,H],$$
that is  $[A,H]\leq  A\cdot (x-1)$. Now, since 
$A/[A,H]$ is central in $H/[A,H]$ and $H/A$ is cyclic, the quotient $H/[A,H]$ is abelian; so $\gamma_2(H)\leq  [A,H]$. This proves the claim for $n=1$. Next,  assume (\ref{Lower series of the cyclotomoc group}) holds for some $n\geq 1$, and let    $a\in \gamma_{n+1}(H)$.  As above, we may write $[a,h]=a'\cdot(x-1)$, with $a'\in \gamma_{n+1}(H)$.  By assumption,  $a'\in A\cdot (x-1)^n$,  so $[a,h]\in A\cdot(x-1)^{n+1}$, which proves that  $\gamma_{n+2}(H)\subseteq  A\cdot (x-1)^{n+1}$; the claim follows.

Recall that the  $\mathfrak{O}$-module structure on
$A$ is defined by $a\cdot f(\pi)=a\cdot f(x-1)$, for $a\in A$ and $f\in \mathbb{Z}_p[X]$;  whence $A\cdot (x-1)^{n}=A\cdot \pi^{n}=A\cdot \mathfrak{P}^{n}$,
 and by (\ref{Lower series of the cyclotomoc group}), $\gamma_{n+1}(H)=A\cdot \mathfrak{P}^{n}$, for all  $n\geq 1$; which proves (i).

By assumption, the $\mathfrak{O}$-module $A$ satisfies 
$$A\cong \bigoplus_{i=1}^{s} \mathfrak{O}/\mathfrak{P}^{n_i},$$
and since $A\neq 0$, it follows that $n_1$ is the smallest positive  integer such that $A\cdot \mathfrak{P}^{n_1}=0$; so by (i),  $H$ has nilpotency class $n_1$.  Moreover, if we write $n_1=l(p-1)+r$, with $0< r \leq p-1$;  then  by Corollary \ref{Structure of om as abelian groups} we have
	$$\mathfrak{O}/\mathfrak{P}^{n_1}\cong (\mathbb{Z}/p^{l+1}\mathbb{Z})^{r} \oplus (\mathbb{Z}/p^{l}\mathbb{Z})^{p-r-1}. $$
Plainly, $\mathbb{Z}/p^{l+1}\mathbb{Z}$ is a maximal cyclic $p$-group involved in the  decomposition of $A$ as an abelian group; thus $A$ has exponent $p^{l+1}$, and obviously $l+1=\lceil \dfrac{n_1}{p-1} \rceil$.  This proves (ii).  

For (iii), observe that for every  $n\geq 1$ we have  $p^n \mathfrak{O}=\mathfrak{P}^{n(p-1)}$; so in additive notation,  $p^nA=A\cdot p^n\mathfrak{O}=A\cdot \mathfrak{P}^{n(p-1)}$;  the result is immediate now by  (i).  

Next, observe that for every integer $j\geq 2$, we have
$$\gamma_j(H)=A\cdot \mathfrak{P}^{j-1}\cong \bigoplus_{i=1}^{s} \mathfrak{P}^{\min\{n_i,j-1\}}/\mathfrak{P}^{n_i}.$$
Therefore,
$$\gamma_j(H)/\gamma_{j+1}(H)\cong \bigoplus_{i=1}^{s} \mathfrak{P}^{\min\{n_i,j-1\}}/\mathfrak{P}^{\min\{n_i,j\}}.$$
For $n_i\leq j-1$, the quotient $\mathfrak{P}^{\min\{n_i,j-1\}}/\mathfrak{P}^{\min\{n_i,j-1\}}$ vanishes; hence
$$\gamma_j(H)/\gamma_{j+1}(H)\cong \bigoplus_{ n_i\geq j } \mathfrak{P}^{j-1}/\mathfrak{P}^{j}\cong\bigoplus_{n_i\geq j }\mathfrak{O}/\mathfrak{P}.$$
Now, as $\mathfrak{O}/\mathfrak{P}$ is a cyclic group of order $p$,  $\gamma_j(H)/\gamma_{j+1}(H)$ is elementary abelian, and its rank $d_j$ is equal to the cardinality of the set $\{i\mid n_i\geq j\}$. Also, we have $A/\gamma_{2}(H)\cong \bigoplus_{i=1}^{s} \mathfrak{O}/\mathfrak{P}$,  so 
$A/\gamma_{2}(H)\cong (\mathbb{Z}/p\mathbb{Z})^{s}$. Thus, $A/\gamma_{2}(H)$ is elementary abelian of rank $d_1=\ell(\lambda)$.  Moreover, 
$$\vert A:\gamma_{p}(H)\vert =\vert A:\gamma_{2}(H)\vert \prod_{j=2}^{p-1} \vert \gamma_{j}(H):\gamma_{j+1}(H)\vert;$$  
it follows from (iii) and (iv) that
$p^{d(A)}=p^{\sum_{i= 1}^{p-1} d_i}$, which proves (iv), and completes the proof of Theorem \ref{Main 2}.

\section{Cohomology and non-inner automorphisms}\label{Cohom and p-Aut}

In this section, $Q$ denotes a finite $p$-group, and $A$  a $Q$-module which is also a finite $p$-group.  We shall denote by $\hat{H}^{n}(Q,A)$, $n\in \mathbb{Z}$,  the Tate cohomology groups of $Q$ with coefficients in $A$  (cf. e.g., \cite[Chap. VIII, \S1]{Serre}). Note that $$\hat{H}^{0}(Q,A)=A^{Q}/A^{\tau}, \quad \mbox{and } \quad \hat{H}^{-1}(Q,A)=\ker \tau/[A,Q];$$
 where $A^{Q}$ is the submodule of  fixed points in $A$, and $\tau:A\to A$ is the norm (or trace) map induced by $Q$ (that is,  $a^{\tau}=\sum_{x\in Q}a\cdot x$, for all $a\in A$).

Recall that the $Q$-module $A$ is  said to be  \textit{cohomologically trivial}  if  $\hat{H}^n(S,A)$ vanishes, for all $S\leq Q$  and $n\in \mathbb{Z}$.  By a result of   Gasch\"{u}tz \cite{Gasch1965}, for $A$ to be  cohomologically trivial, it suffices that $\hat{H}^n(Q,A)=0$ for some  $n\in \mathbb{Z}$. 
Note that the latter result was established independently by Uchida \cite{Uchida}; we shall refer to it as the Gasch\"{u}tz-Uchida theorem.

The following result, due to P. Schmid (see \cite[Proposition 1]{Schmid}), is needed to prove Theorem \ref{The cohomological auxilliary result}.

\begin{lemma}\label{Schmidcentralizer}
	If $A\neq 0$ is  cohomologically trivial as  $Q$-module, then for every non-trivial subgroup $H$ of $Q$, we have $C_Q(A^H)=H$. 
\end{lemma}

\subsection*{Proof of Theorem \ref{The cohomological auxilliary result}}

Let $x\in Q$ be of order $p$, and let $B=A\cdot (x-1)$. Clearly, $B\cdot (1+\cdots+x^{p-1})=0$, so we can view $B$ as a module over $\mathfrak{O}=\mathbb{Z}_p[X]/(g)$.  Write
	$$B\cong \bigoplus_{i=1}^{s} \mathfrak{O}/\mathfrak{P}^{n_i}.$$
 	If $n_1\geq p-1$, then by Corollary \ref{Structure of om as abelian groups}, $\mathfrak{O}/\mathfrak{P}^{n_1}$ yields an abelian $p$-group of rank $p-1$;  whence $A$ cannot be generated by less than $p-1$ elements, a contradiction.  This shows that $n_1\leq p-2$.  Now,  by Theorem \ref{Main 2}(ii) we have $B\cdot (x-1)^{p-2}=0$   and $pB=0$, that is to say $A\cdot (x-1)^{p-1}=0$ and $pA\cdot (x-1)=0$.  Consider the norm map $\tau_x:A\to A$  induced by $\langle x\rangle$ (that is $a^{\tau_x}= a\cdot (1+\ldots+x^{p-1})$, for all $a\in A$).  Using the  identity
 	$$1+X+\ldots+X^{p-1}=\sum_{i= 1}^{p}\binom{p}{i}(X-1)^{i-1},$$
 	it follows that  $$a^{ \tau_x}= \sum_{i= 1}^{p}\binom{p}{i}a\cdot (x-1)^{i-1}=pa.$$  
 Assume for a contradiction that $A$ is cohomologically trivial. In particular, $\hat{H}^{0}(\langle x\rangle ,A)=A^{\langle x\rangle}/A^{\tau_x}=0$; subsequently $A^{\langle x\rangle}=A^{\tau_x}=pA$, and by
 Lemma \ref{Schmidcentralizer} we have $C_Q(pA)=\langle x \rangle$.    If $y\in Q$ is another elements of order $p$, it follows similarly that $C_Q(pA)=\langle y \rangle$; consequently $Q$ has a unique subgroup of order $p$, so $Q$ is either cyclic or a generalized quaternion.
 If $p=2$, then $A=0$, which is not the case; so we may assume  $p>2$, and 
 $Q$ is then cyclic,  a contradiction.  Theorem \ref{The cohomological auxilliary result} follows.

\subsection*{$1$-Cocyles and automorphisms}

Let  $\delta: Q\to A$ be a $1$-cocycle, i.e., a map satisfying $\delta(xy)=\delta(x)\cdot y+\delta(y)$, for all $x,y\in Q$.    Given $x\in G$, we shall denote by $A_{\delta,x}$  the subgroup  of $A$ generated by the elements $\delta(x^n)$,  $n\geq0$.  It is readily seen that 
$$\delta(x^n)=\delta(x)\cdot (1+x+\cdots+x^{n-1}), \quad \mbox{for all } n\geq 1.$$
It follows in particular that $A_{\delta,x}$ is  the submodule of $A$ (viewed as an $\langle x\rangle$-module)  generated by $\delta(x)$.  Furthermore,  
if $x^p=1$, then $A_{\delta,x}$ can be viewed as a module over $\mathfrak{O}=\mathbb{Z}_p[X]/(g)$.  

\begin{lemma}\label{The structure of the modules generated by image of derivation}
With the above notations, the $\mathfrak{O}$-module 
$A_{\delta,x}$ satisfies
$$A_{\delta,x}\cong \mathfrak{O}/\mathfrak{P}^{n}, \quad \mbox{ for some integer } n\geq 0.$$
  In particular, $A_{\delta,x}$ is elementary abelian if and only if $n\leq p-1$.
\end{lemma}
\begin{proof}
 The  $\mathfrak{O}$-linear map  $\alpha\mapsto \delta(x)\cdot \alpha$, from $\mathfrak{O}$ to $A_{\delta,x}$, is onto since $A_{\delta,x}$ is generated by $\delta(x)$ .   Moreover, the kernel of this map is a non-zero ideal of $\mathfrak{O}$ (as $A_{\delta,x}$ is finite), so it has the form $\mathfrak{P}^{n}$ for some integer $n\geq 0$, which proves the first assertion. The second assertion is immediate by Corollary \ref{Structure of om as abelian groups}.  
\end{proof}

Assume $A$ is a normal abelian subgroup of $p$-group $G$, and $Q=G/C_G(A)$.  Then $A$ is a $Q$-module under the action  defined by $a\cdot \bar{g} =a^{g}$, for $a\in A$ and  $g\in G$ (of course, $\bar{g}$ denotes the canonical image of $g$ in $Q$).  Let $Z^1(Q,A)$  denote  the abelian group formed by the $1$-cocycles  $\delta:Q\rightarrow A$. 
Recall that an element $\delta\in Z^1(Q,A)$ is called a \textit{$1$-cobord} if there exists $a\in A$ such that $\delta(x)=a(x-1)$, for all $x\in Q$; these $1$-cobords form a subgroup $B^1(Q,A)$ of $Z^1(Q,A)$, and the quotient $Z^1(Q,A)/B^1(Q,A)$ coincides with the first cohomology group $\hat{H}^1(Q,A)$.  

For every $\delta\in Z^1(G,A)$, we can define an automorphism   $\varphi_{\delta}:G\to G$ by
$$
g^{\varphi_{\delta}}=g\delta(\bar{g}), \quad \mbox{for all } g\in G.  
$$
The map $\delta \mapsto \varphi_{\delta}$  defines clearly a one-to-one group homomorphism, and its  image is equal to the subgroup $\mathrm{Aut}(G;A)$ formed by the automorphisms $\sigma $ of $G$ which act trivially on  $C_G(A)$ and $G/A$ (that is to say  $x^{\sigma}=x$ for all $x\in C_G(A)$, and $g^{-1}g^{\sigma}\in A$ for all $g\in G$).  Write $\mathrm{Inn}(G;A)$ for the inner automorphisms of $G$ lying in $\mathrm{Aut}(G;A)$.  Using the multiplicative notation, a $1$-cobord $\delta$ induced by some $a\in A$ may be written as  $\delta(\bar{g})=a^g a^{-1}$, for all $g\in G$; thus $\varphi_{\delta}$ is the inner automorphism of $G$ induced by $a^{-1}$.  This shows in particular that $B^1(Q,A)^{\varphi}\subseteq \mathrm{Inn}(G;A)$.  
\begin{lemma}\label{From derivations to automorphisms}
Under the above notations, the equality $B^1(Q,A)^{\varphi}= \mathrm{Inn}(G;A)$ holds if and only if  $A=C_G(C_G(A))\cap Z_G(A)$, where $Z_G(A)$ is the subgroup satisfying 
	$Z_G(A)/A=Z(G/A)$.  In this case, $\varphi$ induces an isomorphism from $\hat{H}^1(Q,A)$ onto $ \mathrm{Aut}(G;A)/ \mathrm{Inn}(G;A)$.  
\end{lemma}   
The above result is well-known in the literature (it is usually formulated as $C_G(C_G(A))=A$ is a sufficient condition for $\varphi $ to apply $B^1(Q,A)$  onto  $\mathrm{Inn}(G;A)$); the proof is straightforward and is left to the reader.

For the  proof of Theorem  \ref{App to the noninner Conj}, we need a result by Deaconescu and Silberberg \cite{Deac}, asserting  that  the non-inner conjecture  holds in the case where $C_G(Z(\Phi(G)))\neq \Phi(G)$ (a slight improvement of the latter can be found in Ghoraishi \cite{Gho2014}).

\subsection*{Proof of Theorem \ref{App to the noninner Conj}}
Set $A=Z(\Phi(G))$.  As we have mentioned above, we may suppose  $C_G(A)=\Phi(G)$, and so $Q=G/\Phi(G)$.  Theorem \ref{The cohomological auxilliary result} shows that the $Q$-module $A$ is not   cohomologically trivial,  so by the Gasch\"{u}tz-Uchida theorem we have $\hat{H}^1(Q,A)\neq 0$.  Now, Lemma \ref{From derivations to automorphisms} guarantees that $\mathrm{Aut}(G;A)$ contains non-inner automorphisms, and our theorem will follow once we show that $\mathrm{Aut}(G;A)$ has exponent $p$.  As $\mathrm{Aut}(G;A)$ is isomorphic to $Z^1(Q,A)$, we need only to show that the latter has exponent $p$.  Let $\delta \in Z^1(Q,A)$ and $x\in Q$.  By Lemma \ref{The structure of the modules generated by image of derivation}, we have $A_{\delta,x}\cong \mathfrak{O}/\mathfrak{P}^{n}$ for some  integer $n$.  Our assumption on $A$ guarantees that $A_{\delta,x}$ can be generated by $p-2$ elements, so $n\leq p-2$ by Corollary \ref{Structure of om as abelian groups}.  It follows that $A_{\delta,x}$ is elementary abelian; thus $p\delta(x)=0$ for all $x\in Q$, that is $p\delta=0$, as claimed.

\section{Further results and questions}\label{Concl Sec}  
\subsection*{Extensions of discrete valuation rings}
Let $\mathfrak{o}$ be a discrete valuation ring with maximal ideal  $\mathfrak{p}$, and $e$ be a postive integer.  There is a systematic way  for constructing discrete valuation rings $\mathfrak{O}$ that extend  $\mathfrak{o}$ and satisfy:  (i) $\mathfrak{O}$ is a finitely generated 
$\mathfrak{o}$-module, and (ii) the maximal ideal $\mathfrak{P}$ of $\mathfrak{O}$ satisfies $\mathfrak{p}\mathfrak{O}=\mathfrak{P}^e$.  Indeed, consider any Eisenstein polynomial $f=X^e+a_{e-1}X^{e-1}+\cdots+a_0$ over 
$\mathfrak{o}$ (so $a_{e-1},\ldots a_0 \in \mathfrak{p}$ and $a_0 \notin \mathfrak{p}^2$), and let $\mathfrak{O}= \mathfrak{o}[X]/(f)$.  Then,  $\mathfrak{O}$ is a discrete valuation ring which satisfies the previous conditions (cf. \cite[Chap. I, \S 6]{Serre}).  

Let us include a proof for the reader convenience.   First,  
 identify $\mathfrak{o}$ to its canonical image in  $\mathfrak{O}$.   It is clear that $\bar{1}, \bar{X}, \ldots, \bar{X}^{e-1}$ form a basis for the $\mathfrak{o}$-module $\mathfrak{O}$ (with $\bar{X}$ denotes the canonical image of $X$ in $\mathfrak{O}$).   Set $\mathfrak{P}=(\bar{X})$.   By  assumption, $\mathfrak{p}=(a_0)$, and since
 $a_0=-(\bar{X}^{e}+\cdots+a_1 \bar{X})$, we have
 $\mathfrak{p}\mathfrak{O}\subseteq \mathfrak{P} $.   Next, observe that  reduction modulo $\mathfrak{p}$ induces an onto ring morphism $\varphi:\mathfrak{o}[X]\to (\mathfrak{o}/\mathfrak{p})[X]$ which sends $f$ to $X^e$; whence we have an onto morphism $\mathfrak{O}\to (\mathfrak{o}/\mathfrak{p})[X]/(X^e)$, whose kernel  is $\mathfrak{p}\mathfrak{O}$. It follows that $\mathfrak{P}$ is the unique maximal ideal of $\mathfrak{O}$ containing $\mathfrak{p}\mathfrak{O}$.   If $\mathfrak{m}$ is a maximal ideal of $\mathfrak{O}$, then  $\mathfrak{p}\mathfrak{O} \subseteq \mathfrak{m}$ (otherwise, we would have $\mathfrak{p}\mathfrak{O}+ \mathfrak{m}=\mathfrak{O}$, so by Nakayama's Lemma,   $\mathfrak{m}=\mathfrak{O}$, a contradiction); thus $\mathfrak{m}=\mathfrak{P}$. This shows that $\mathfrak{O}$ is a local ring, which is clearly Noetherien,  and its maximal ideal $\mathfrak{P}$ is generated by a non-nilpotent  element; thus $\mathfrak{O}$ is a discrete valuation ring (cf. \cite[Proposition 2, Chap. I]{Serre}).  The fact that $\mathfrak{p}\mathfrak{O}=\mathfrak{P}^e$ is immediate by virtue of the isomorphism 
 $\mathfrak{O}/\mathfrak{p}\mathfrak{O}\cong k[X]/(X^e)$.  

\subsection*{Higher cyclotomic polynomials}

Take $\mathfrak{o}=\mathbb{Z}_p$, and  $g_m(X)=\Phi_{p^m}(X+1)$, where $\Phi_{p^m}$ denotes the $p^m$th cyclotomic polynomial.  Plainly,  $g_m(X)$  is  an Eisenstein polynomial over $\mathbb{Z}_p$, and 
$\mathfrak{O}_m=\mathbb{Z}_p[X]/(g_m)$ coincides with the ring of integers of the cyclotomoic field $\mathbb{Q}_p[\zeta_{p^m}]$, where $\zeta_{p^m}$ is a primitive $p^{m}$th root of the unity. By the above paragraph,  we have  $e=(p-1)p^{m-1}$, that is  the maximal ideal $\mathfrak{P}$ of $\mathfrak{O}_m$ satisfies    $p\mathfrak{O}_m=\mathfrak{P}^{ (p-1)p^{m-1}}$.  

Let $A$ be a finite  abelian $p$-group, and $\langle x \rangle$  a cyclic group acting on $A$.  We may view $A$ as a module over the polynomial rings $\mathbb{Z}_p[X]$, with $X$ acts on $A$ as the endomorphism $x-1$. Obviously,  $A$  satisfies $A\cdot g_m=0$, if and only if 
\begin{equation}\label{Condition for Om mod}
A\cdot (1+x^{p^{m-1}}+\cdots+x^{p^{m-1}(p-1)})=0.
\end{equation}
This condition amounts to saying   that
$x^{p^{m-1}}$ acts on $A$ as a splitting automorphism of order $p$.   It follows that every $A$  admitting such an automorphism has a natural structure of a module over  $\mathfrak{O}_m=\mathbb{Z}_p[X]/(g_m)$.  Applying Theorem \ref{Main 1} with $\mathfrak{o}=\mathbb{Z}_p$ and 
$\mathfrak{O}= \mathfrak{O}_m$,   yields a more general version of Corollary \ref{Structure of om as abelian groups} (the latter follows by taking $m=1$).  For instance, if $A\cong \mathfrak{O}_m/\mathfrak{P}^n$ as an 
$\mathfrak{O}_m$-module, then the structure of $A$ as an abelian group is given by
$$A\cong 
 (\mathbb{Z}/p^{l+1}\mathbb{Z})^{r} \oplus (\mathbb{Z}/p^{l}\mathbb{Z})^{(p-1)p^m-r}, $$
	where $n=l(p-1)p^{m-1}+r$, and $0< r \leq (p-1) p^{m-1}$.  

We have also an obvious analogue of Theorem  \ref{Main 2}. Assume that $\langle x \rangle$ acts on $A$ in such a way that $x^{p^{m-1}}$ induces a splitting automorphism of order $p$ of $A$ (in other words, $x$ satisfies the condition (\ref{Condition for Om mod})).  Let $H$ be any extension of $A$ by $\langle x \rangle$, and $\lambda=(n_1,\ldots,n_s)$ be the partition corresponding to $A$ viewed as a module over  $\mathfrak{O}_m=\mathbb{Z}_p[X]/(g_m)$.  As in the proof of Theorem \ref{Main 2}, we have
$$ \gamma_{n+1}(H)=A\cdot \mathfrak{P}^{n}, \mbox{ for all } n\geq 1.$$
It follows similarly that:
	\begin{itemize}
		\item[(a)]  $H$ has nilpotency class $n_1$, and $A$ has exponent  $p^{l}$, where $l$ is the smallest integer $\geq \dfrac{n_1}{p^{m-1}(p-1)}$. 
		\item[(b)]  $A^{p^n}=\gamma_{n(p-1)p^{m-1}+1}(H)$, for all $n\geq 1$.
		\item[(c)] For   $j\geq 2$, the factor group  $\gamma_{j}(H)/\gamma_{j+1}(H)$ is elementary abelian of rank $d_j$, where   $(\ell(\lambda),d_2,d_3,\ldots)$ is the partition conjugate to $\lambda$. 
		\item[(v)] $A/\gamma_{2}(H)$ is elementary abelian of rank $d_1=\ell(\lambda)$, and   $$d(A)= \sum_{j=1}^{(p-1)p^{m-1}} d_j.$$
	\end{itemize}
Of course, one can cover Theorem \ref{Main 2} by taking $m=1$.

\subsection*{The number of induced modules}
Let $\lambda$ be a partition, and $M$ be the corresponding $\mathfrak{O}$-module.  As $\mathfrak{O}/\mathfrak{P}$ is an $\mathfrak{o}$-module of length $d$, the $\mathfrak{o}$-module 
$\rho_{\ast}M$ has length $\vert \lambda \vert d$.   On the other hand, Theorem \ref{Main 1} shows that $\rho_{\ast}M$ has length $\sum_{i\geq 1} if_{e,i}(\lambda)d$; therefore,
\begin{equation}\label{Main identity for $f_i$}
\sum_{i\geq 1} if_{e,i}(\lambda)=\vert\lambda\vert.
\end{equation} 
One can give a direct proof to the above identity.  First,  the formulae  (\ref{The relation between the weight and the sum of the parts})  and (\ref{The invariants of the associated abelian group}) yield  at once
\begin{equation}\label{Alternative formula for fei}
f_{e,i}(\lambda)=\ell (\lambda_{e,i+1})(i+1)e-\ell(\lambda_{e,i}) (i-1)e-\vert \lambda_{e,i+1}\vert+\vert\lambda_{e,i}\vert.
\end{equation}
Next, observe that in the sum $\sum_{i\geq 1} if_i(\lambda)$,  the term   $\sum_{i\geq 1}i(i+1)\ell(\lambda_{e,i+1})e$ vanishes with  $-\sum_{i\geq 1}i(i-1)\ell(\lambda_{e,i})e$; whence
\begin{align*}
	\sum_{i\geq 1} if_{e,i}(\lambda)& = \sum_{i\geq 1} i \vert \lambda_{e,i}\vert -\sum_{i\geq 1}i\vert \lambda_{e,i+1}\vert \\
	& = \sum_{i\geq 1} i \vert \lambda_{e,i}\vert -\sum_{i\geq 1}(i-1)\vert \lambda_{e,i}\vert\\
	& =\sum_{i\geq 1} \vert\lambda_{i}\vert\\
	&  = \vert\lambda\vert, 
\end{align*}
as claimed.   

\medskip 
Let $\mathcal{P}$ be the set of all partitions, and  $\mathcal{P}_n$ be the set of partitions of $n\geq 1$.  It follows that the map $f_e:\mathcal{P}\to \mathcal{P}$, defined by
$f_e(\lambda)=(1^{f_{e,1}(\lambda)}\ldots i^{f_{e,1}(\lambda)}\ldots)$, sends $\mathcal{P}_n$ to itself (that is $f_e(\mathcal{P}_n)\subseteq \mathcal{P}_n$).   To simplify, let $f_e(n)$ denote the cardinality of the set $f_e(\mathcal{P}_n)$.  Clearly, $f_e(n)$ is the number of $\mathfrak{o}$-modules (up to isomorphism) of length $n$ their structure is induced by an action of $\mathfrak{O}$.  In particular, 
 for $e=(p-1)p^{m-1}$ ($p$ is a prime), $f_{e}(n)$ gives the number of abelian groups of order $p^n$ that admit an automorphism whose $p^{m-1}$th power is splitting automorphism of order $p$. 
 
 	\begin{prob}\label{The proportion of Abelian groups having splitting aut } For every $e\geq 2$, determine the map $n\mapsto  f_e(n)$.  In particular, determine the behavior of the proportion  $ f_e(n)/\mathrm{p}(n)$, as $n\to +\infty$ (where $\mathrm{p}(n)=	\vert \mathcal{P}_n\vert$.)
 \end{prob}
 For example, $f_{e}(n)=1$ if $n\leq e$ (as every partition $\lambda$ of such an $n$  gives the $\mathfrak{o}$-module    $(\mathfrak{o}/\mathfrak{p})^{n}$).  Assume $n$ lies in the interval $ \left]e, 2e\right]$, and  $\lambda$ is a partition of $n$.  If $\lambda_{e,2}$ is empty, then $\lambda$ gives the $\mathfrak{o}$-module    $(\mathfrak{o}/\mathfrak{p})^{n}$.  If  $\lambda_{e,2}$  is not empty, then it contains exactly one component $n_1$ say, and  $\lambda$ yields the $\mathfrak{o}$-module  $(\mathfrak{o}/\mathfrak{p}^2)^r\oplus (\mathfrak{o}/\mathfrak{p})^{n-2r}$, where $r=n_1-e$. Since the possible values of $r$ are $e+1,\ldots,n$,  it follows that $f_{e}(n)=n-e+1$ for $e<n\leq 2e$.  One can compute $f_e(n)$ similarly in the intervals $ \left]2e, 3e\right]$, $ \left]3e, 4e\right]$, etc.,  though this procedure becomes less and less adequate. It also seems that the values of $f_e(n)$ are related to that of  $f_{e'}(n)$ for $e$ dividing $e'$.
 
\subsection*{On the structure of some $p$-groups}
 
 Assume $A$ is  an abelian finite $p$-group, and $x$ is an  automorphism of $A$ of order $p$.  Obviously, $x$ acts on $A\cdot (x-1)$ as a splitting automorphism. Corollary \ref{Structure of om as abelian groups} and Theorem \ref{Main 2} yield then restrictions on the structure of $A\cdot (x-1)$, and  on that of $A/A^{\langle x \rangle}$  as well (where $A^{\langle x \rangle}$ is the subgroup of  fixed elements under $x$).   Note that an answer to Problem \ref{The proportion of Abelian groups having splitting aut } would shed more light on these restrictions.  
 
 The first possible application of the latter would be an effective classification of the non-abelian finite $p$-groups $G$ that admit an abelian subgroup $A$ of index $p$;  such a classification would allow us to compute easily the number  of such $G$'s having order $p^n$, for every $n\geq 1$.  
 
 One may start with the case where there exists an elements $x\in G\setminus A$ which acts (by conjugation) on $A$ as a splitting automorphism of order $p$.  Fixing $A$ and the action of $Q=\langle x\rangle$,  the    extensions of $A$ by $Q$, up to equivalence,  correspond to the elements of the second cohomology group $\hat{H}^2(Q,A)$, which is easily seen to be isomorphic to  $(\mathbb{Z}/p\mathbb{Z})^{\ell(\lambda)}$, where $\lambda$ is the partition corresponding to $A$ as an $\mathfrak{O}$-module.  

One may also use  the forgoing results in classifying metabelian  $p$-groups of  maximal class. Such a classification was already  established by Miech \cite{Miech1}.  A simplified proof  appeared later in \cite{Miech2}, although one may require a more transparent  proof that allows us to find the number of such $p$-groups, for every given order,  in a straightforward manner. Note that Theorem \ref{Main 2}, shows at least that 
$\gamma_3(G)$ has the form $\mathfrak{O}/\mathfrak{P}^n$, for every such a $p$-group $G$. 

Finally, note that if $M$ is a maximal subgroup of an arbitrary finite $p$-group $G$ such that $Z(G)\leq M$, then the restrictions yielded by Corollary \ref{Structure of om as abelian groups} and Theorem \ref{Main 2} apply to $Z(M)/Z(G)$. Indeed, every $x\in G\setminus M$ acts by conjugation on $[Z(M),x]$  as a splitting   automorphism of order $p$, and obviously, $Z(M)/Z(G)\cong [Z(M),x]$.

\section*{Acknowledgment}
We are grateful to Naomi Jochnowitz and Evgeny Khukhro for their interest and for several invaluable discussions.
% \input{preliminary.tex}

% BIBLIOGRAPHY

\end{document}